\documentclass[12pt,reqno]{amsart}
\usepackage{amsfonts,amsmath,amssymb,amsthm,mathdots,mathtools}
\usepackage{dsfont,enumerate,float,graphicx,hyperref,makecell,pgfplots,perpage,thmtools,url}
\usepackage[margin=2.2cm]{geometry}
\usepackage{youngtab}

\usepackage{orcidlink}

\theoremstyle{plain}
\newtheorem{thm}{Theorem}[section]

\newtheorem{cor}[thm]{Corollary}
\newtheorem{lemma}[thm]{Lemma}

\newtheorem{prop}[thm]{Proposition}

\theoremstyle{remark}

\numberwithin{equation}{section}

\allowdisplaybreaks
\MakePerPage{footnote}  
\newcommand\topalign[1]{%
	\setbox0\hbox{#1}%
	\raisebox{\dimexpr-\ht0+\dp0\relax}{\usebox0}}

\title{A note on Diagonal sequences of integer partitions}
\author{Michael Neubauer \orcidlink{0000-0001-7514-7566}}
\author{Harmony Vargas \orcidlink{0009-0002-7364-5216} }
\email{michael.neubauer@csun.edu}
\email{hamrmony.vargas.73@my.csun.edu}
\address{Department of Mathematics, California State University, Northridge, CA 91330, USA}
\date{October 22, 2024}

\thanks{}

\begin{document}
	
	\setlength{\parindent}{0in} \setlength{\parskip}{.2in} 
	
	\begin{abstract}
		Let \(\mathcal{P}(n)\) be the set of partitions of the positive integer \(n\). For \(\alpha=(\alpha_1,...,\alpha_t) \in \mathcal{P}(n)\) define the diagonal sequence \(\delta(\alpha)=(d_k(\alpha))_{k \geq 1}\) via \( d_k(\alpha) = \big\lvert \{ i \, \rvert \, 1 \leq i \leq k \mbox{ and  } \alpha_i + i- 1\geq k \}  \big\rvert.\) We show that the set of all partitions in \(\mathcal{P}(n)\) with the same diagonal sequence is a partially ordered set under majorization with unique maximal and minimal elements and we give an explicit formula for the number of partitions with the same diagonal sequence.   
	\end{abstract}
	\maketitle

	\section{Introduction}

	Let \(\mathcal{ P}(n)\) be the set of partitions of the natural number \(n\), i.e., for \(\alpha=(\alpha_1,...,\alpha_l) \in \mathcal{ P}(n)\) we assume \(\alpha_1 \geq \alpha_2 \geq \cdots \geq \alpha_l \geq 1\), \(\alpha_i \in \mathbb{N}\), and \(\sum_{i=1}^l \alpha_i=n\). Let \(\mathcal{P}(n,k)\) be the partitions of \(n\) with exactly \(k\) non-zero parts and let \(\mathcal{P}(n,k)^*\) be the set of conjugates of \(\mathcal{P}(n,k)\), i.e.,  \(\mathcal{P}(n,k)^*\) is the set of all partitions of \(n\) that have largest part equal to \(k\)..   
	
	\textbf{Definition} 
	\textit{For a partition \(\alpha= (\alpha_1, \cdots,\alpha_l) \in \mathcal{ P}(n)\)  define its diagonal sequence \(\delta(\alpha)=(d_k)_{k \geq 1}\) via
		\begin{equation}
			d_k=d_k(\alpha) = \big\lvert \{ i \, \rvert \, 1 \leq i \leq k \mbox{ and  } \alpha_i + i- 1\geq k \}  \big\rvert.
	\end{equation} }
	
	Since only finitely many \(d_k\) are positive we may omit writing trailing zeroes for \(\delta(\alpha)\) .
	
	Using Young diagrams we can visualize diagonal sequences in two ways. For example, the partition \(\alpha=(7,7,4,1,1,1) \in \mathcal{ P}(21)\) has diagonal sequence \(\delta(\alpha) = (1,2,3,4,4,4,2,1).\)
	
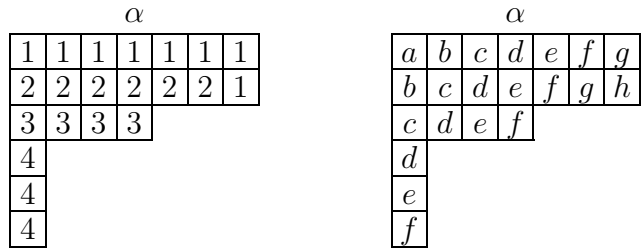
\begin{center}
		\begin{tabular}[t]{cccccc}
			\multicolumn{1}{c}{$\alpha$} & & & & &  \multicolumn{1}{c}{$\alpha$}  \\ 
			$\young(1111111,2222221,3333,4,4,4)$ & & & & &
			$\young(abcdefg,bcdefgh,cdef,d,e,f)$  \\ 
		\end{tabular}
	\end{center}

	In the  Young diagram on the left the numbers in the squares indicate the position of the square on the respective diagonal. In the Young diagram on the right all squares on the same diagonal are marked by the same letter.

	\textbf{Definition} \textit{Let \(\Delta(n) = \{ \delta(\alpha) \,| \, \alpha \in \mathcal{P}(n) \}\).}

	This defines \(\delta\) as a map from \(\mathcal{P}(n)\) to \(\Delta(n)\) with \(\alpha \rightarrow \delta(\alpha)\). The map \(\delta\) is surjective by definition and it is not injective as is easy to see. Different partitions may have the same diagonal sequence. If we move a square of the Young diagram on an upper left to lower right diagonal in such a way that we end up with another Young diagram then the two partitions will have the same diagonal sequence. E.g., the partition \(\beta = (8,6,2,2,2,1) \in \mathcal{ P}(21)\)  
	
	\begin{center}
		\begin{tabular}[t]{cccccc}
			\multicolumn{1}{c}{$\beta$} & & & & &  \multicolumn{1}{c}{$\beta$}  \\ 
			$\young(11111111,222222,3333,444)$ & & & & &
			$\young(abcdefgh,bcdefg,cdef,def)$  \\ 
		\end{tabular}
	\end{center}
	has the same diagonal sequence as the partition \(\alpha=(7,7,4,1,1,1)\) above. The h-square in row 2 of \(\alpha\) moved to the last position in row 1 of \(\beta\) while the e-square in row 5 of \(\alpha\) moved to the second position in row 4 of \(\beta\) and the f-square in row 6 of \(\alpha\) moved to the fourth square in row 4 of \(\beta\). The reader may have noticed that there are more diagonal moves that yield different partitions with the same diagonal sequence. Theorem \ref{theorem:anylength} in Section \ref{sec:main} gives and explicit expression for the number of partitions with the same diagonal sequence. The transition from \(\alpha\) to \(\beta\) moved squares up along the diagonals. In fact, the squares of the Young diagram of \(\beta\) are as high up along their respective diagonals as possible.  We explore this idea in Section \ref{section:majorization}.   
	
	We can define an equivalence relation on \(\mathcal{P}(n)\) via \(\alpha \sim \beta\) if and only if \(\delta(\alpha)=\delta(\beta)\). The equivalence class of a partition \(\alpha \in \mathcal{P}(n)\) is characterized by the invariant \(\delta(\alpha)=d\). For \(d \in \Delta(n)\) define \([ \, d \, ] =\delta^{-1}(d)\subseteq \mathcal{P}(n)\).
	
	In particular, the partition \(\alpha \in \mathcal{P}(n)\) and its conjugate \(\alpha^*\) have the same diagonal sequence and hence belong to the same equivalence class, i.e., \(\alpha \sim \alpha^*\) and \(\delta(\alpha) = \delta(\alpha^*)\). The conjugate partition is obtained by reflecting the Young diagram about the upper left to lower right diagonal which maps each upper right to lower left diagonal onto itself.

	Diagonal sequences were used in \cite{AFNW} to find graphs with maximal sums of squares of degrees and in \cite{N} to find bipartite graphs with maximal sums of squares of degrees. 
	
	In the next section we prove properties of \(\delta(\alpha)\) in preparation for the main results of Sections \ref{section:majorization} and \ref{sec:main}. In Section \ref{section:majorization} we show that \([\,d\,]\) is a partially ordered set under majorization with unique maximal and unique minimal element. In Section \ref{sec:main} we give an explicit formula for the size of \(\delta(\alpha)=[\,d\,]\). In Section \ref{sec:examples} we list the 36 partitions in the equivalence class of \(\alpha=(7,7,4,1,1,1)\), i.e., we list all \(\alpha \in \mathcal{ P}(21)\) with \(\delta(\alpha)=(1,2,3,4,4,4,2,1)\). %Section \ref{section:arrangements} is an add-on that generalizes a result from Section \ref{sec:main}.

	\section{Properties of \(\delta(\alpha)\)} \label{sec:properties}
	
	We first prove that the sequence \(\delta(\alpha)\) increases from \(1\) to a positive integer \(q\) in increments of 1 and then continues in a non-increasing sequence.  
	
	\begin{lemma} \label{lemma:delta}
		Let \(\alpha= (\alpha_1, \cdots,\alpha_l) \in \mathcal{ P}(n)\) with \(\delta(\alpha)=(d_k)_{k \geq 1}\) .                           
		\begin{enumerate}
			\item \(d_{k+1}-d_k \leq 1\) with equality if and only if \(d_{j}=j\) for all \(1 \leq j \leq k+1\).
			\item If \(d_k \geq d_{k+1}\), then \(d_{k+1} \geq d_{k+2}\).
		\end{enumerate}
	\end{lemma}
	
	\begin{proof} By definition \(d_{k+1} \leq d_k+1\).  
		Notice that if for some \(1 \leq j \leq k\) we have \(\alpha_j+j-1<k\), then \[\alpha_{j+1}+(j+1)-1 = \alpha_{j+1}+j \leq \alpha_j+j < k+1. \]
		Hence if \(d_j <j\), then \(d_j\geq d_{j+1}\). 
	\end{proof}
	
	\begin{cor} \label{cor:ds}
		For all partitions \(\alpha \in \mathcal{ P}(n)\) there exist unique integers \(q >0\) and \(s_1, \cdots ,s_q \geq 0\) such that 
		\begin{equation} \label{eqn:ds}
			\delta(\alpha) = \left( 1,2, \cdots ,q-1,q,q^{(s_q)},(q-1)^{(s_{q-1})}, \cdots ,1^{(s_1)} \right).
		\end{equation}
		with 
		\begin{equation}
			\dfrac{q(q+1)}{2}+\sum_{k=1}^{q} k \, s_k = n.
		\end{equation}
	\end{cor}
	
	Let \(d=(d_1,d_2,\cdots,d_L) \in \Delta(n)\). We characterize two special elements \([ \, d \, ] \). The notation \(\overline{\alpha}\) and \(\underline{\alpha}\) will become clear in Section \ref{section:majorization}. 
	
	\begin{prop} \label{prop:s}
		Let \(s_1, \cdots,s_q \geq 0\) be integers such that \(\frac{q(q+1)}{2}+\sum_{k=1}^q k \, s_k = n\). Let  \(\overline{\alpha}_i=q-i+1+\sum_{k=i}^{q} s_k\) for \(1 \leq i \leq q\). Then 
		\begin{enumerate}
			\item \(\overline{\alpha} = (\overline{\alpha}_1, \cdots ,\overline{\alpha}_q) \in \mathcal{ P}(n)\),
			\item \(\overline{\alpha}_1 >  \overline{\alpha}_2 > \cdots >\overline{\alpha}_q \),
			\item \( \delta(\overline{\alpha})=\left( 1,2, \cdots ,q-1,q,q^{(s_q)},(q-1)^{(s_{q-1})}, \cdots ,1^{(s_1)} \right) \),
			\item \(s_i=\overline{\alpha}_i-\overline{\alpha}_{i+1}-1\), \(1 \leq i \leq q-1\), and \(s_q=\overline{\alpha}_q-1\).
		\end{enumerate}
	\end{prop}
	
	\begin{proof} 
		By definition \(\overline{\alpha}_1 > \overline{\alpha}_2 > \cdots > \overline{\alpha}_q\) and since \[\sum_{i=1}^q \overline{\alpha}_i = \sum_{i=1}^q (q-i) + \sum_{i=1}^q \, \sum_{k=i}^q s_k = \frac{q(q+1)}{2}+\sum_{k=1}^q k \, s_k = n,\] \(\overline{\alpha}=(\overline{\alpha}_1, \cdots ,\overline{\alpha}_q)\) is a partition of  \(n\), i.e., \(\overline{\alpha} \in \mathcal{ P}(n).\) 
		
		Next, we will show that \begin{equation}
			\delta(\overline{\alpha})=\left( 1,2, \cdots ,q-1,q,q^{(s_q)},(q-1)^{(s_{q-1})}, \cdots ,1^{(s_1)} \right).
		\end{equation}
		
		If \(1 \leq k \leq q\), then  \(\overline{\alpha}_i+i-1= q+\sum_{j=i}^q s_j \geq q \geq k\) for \(1 \leq i \leq k\). Hence \(d_k=k\) for \(1 \leq k \leq q\).
		
		If \(q  < k \leq q+ s_q\), then 
		\begin{equation} \nonumber   
			\overline{\alpha_i}+q-1  = q+\sum_{j=i}^q s_j \geq q+s_q \geq k \mbox{ for } 1 \leq i \leq q.
		\end{equation}
		Hence \(d_k=q\) for \(q  < k \leq q+ s_q\).
		
		Similarly, if \(q+\sum_{j=i}^qs_j < k \leq q+\sum_{j=i-1}^q s_j\), \(i \geq 2\), then \begin{equation} \nonumber
			\begin{split}
				\overline{\alpha}_j +j-1 & = q+\sum_{l=j}^qs_l < k \mbox{ for } j \geq i \\
				\overline{\alpha}_j +j-1 & = q+\sum_{l=j}^q s_l \geq  k \mbox{ for } j < i.\\ 
			\end{split}
		\end{equation} 
		Hence \(d_k=i-1\) for \(q+\sum_{j=i}^qs_j < k \leq q+\sum_{j=i-1}^q s_j\). 
		
		Lastly, we show how to compute the values of \(s_i\), \(1 \leq i \leq q\), given \(\overline{\alpha}\).
		
		Let \(s=(s_1,s_2,, \cdots, s_q)\), \(v=(q,q-1, \cdots,1)\) and 
		%\[s= \begin{bmatrix}
			%    s_1 \\
			%    s_2 \\
			%    s_3 \\
			%    \vdots \\
			%    s_q
			%\end{bmatrix}, \qquad
			%v= \begin{bmatrix}
				%   q \\
				%    q-1 \\
				%    q-2 \\
				%    \vdots \\
				%    1
				%\end{bmatrix}, \qquad
				%\overline{\alpha}= \begin{bmatrix}
					%    \overline{\alpha}_1 \\
					%    \overline{\alpha}_2 \\
					%    \overline{\alpha}_3 \\
					%    \vdots \\
					%    \overline{\alpha}_q
					%\end{bmatrix},  \qquad
					\[T= \begin{bmatrix}
						1 & 1 & 1 & \cdots & 1 \\
						0 & 1 & 1 & \cdots & 1 \\
						0 & 0 & 1 & \cdots & 1 \\
						\vdots & \vdots & \vdots & \ddots & \vdots \\
						0 & 0 & 0 & \cdots & 1 \\
					\end{bmatrix}. \]
					It follows that \(\overline{\alpha}=v+Ts\) and hence \(s=T^{-1}(\overline{\alpha}-v)\). Since
					\[T^{-1} = \begin{bmatrix}
						1 & -1 & 0 & 0 & \cdots & 0 & 0 \\
						0 & 1 & -1 & 0 & \cdots & 0 & 0 \\
						0 & 0 & 1 & -1 & \cdots & 0 & 0 \\
						\vdots & \vdots & \vdots & \ddots & \ddots & \vdots & \vdots \\
						\vdots & \vdots & \vdots & \vdots & \ddots & \ddots & \vdots \\
						0 & 0 & 0 & 0 & \cdots & 1  & -1 \\
						0 & 0 & 0 & 0 & \cdots & 0 & 1
					\end{bmatrix}\] we get \(s_i = \overline{\alpha}_i - \overline{\alpha}_{i+1} - 1\) for \(1 \leq i <q\) and \(s_q = \overline{\alpha}_q - 1.\)
				\end{proof}

				There are some consequences of Proposition \ref{prop:s} that are worth pointing out. 
				
				\begin{cor}
					\(|\Delta(n)|\) is equal to the number of partitions of \(\mathcal{ P}(n)\) with distinct parts which is equal to the number of partitions in \(\mathcal{P}\) with odd parts. 
				\end{cor}
				
				\begin{proof}
					By Proposition \ref{prop:s} for each diagonal sequence \(\delta \in \Delta(n)\) there is a unique partition \(\overline{\alpha}\) with \(\overline{\alpha}_1 > \overline{\alpha}_2> \cdots > \overline{\alpha}_q\). The second part follows from a well-know result that the number of partitions in \(\mathcal{P}(n)\) with distinct parts is equal to the number of partitions in \(\mathcal{P}(n)\) with odd parts. 
				\end{proof}

				Let \(\underline{\alpha}=\overline{\alpha}^*\) denote the conjugate partition of \(\overline{\alpha}\) .
				
				\begin{prop} \label{prop:alphaunderbar}
					If \(\alpha \in \mathcal{P}(n)\) with \(d(\alpha)=\left( 1,2, \cdots ,q-1,q,q^{(s_q)},(q-1)^{(s_{q-1})}, \cdots ,1^{(s_1)} \right)\), then \[\underline{\alpha} = (q^{(s_q+1)}, (q-1)^{(s_{q-1}+1)}, \cdots, 1^{(s_1+1)}).\]
				\end{prop}
				
				\begin{proof}
					Let \(\overline{\alpha}^* =(\overline{\alpha_1}^*,\overline{\alpha_2}^*, \cdots ,\overline{\alpha_t}^*) \). Since \(\overline{\alpha}_1 > \overline{\alpha}_2> \cdots > \overline{\alpha}_q\), \(t=\overline{\alpha_1}\) and for \(\overline{\alpha}_{i+1} < k \leq \overline{\alpha}_{i}\) we have \(\overline{\alpha_k}^*=\underline{\alpha}_k = i\). By Proposition \ref{prop:s}, \(\overline{\alpha}_{i}-\overline{\alpha}_{i+1} = s_i+1\) and the result follows. 
				\end{proof}
				
				\begin{cor} \label{setof integers}
					The multiset of integers of \(\delta(\overline{\alpha})=\delta(\underline{\alpha})\) is equal to the multiset of integers of \(\underline{\alpha}\).  
				\end{cor}  
				
				\begin{cor} \label{cor:twoalphasequal}
					Assume \(\alpha = (\alpha_1,\alpha_2, \cdots,\alpha_t) \in \mathcal{ P}(n)\). If \(\alpha \neq \overline{\alpha}\), then \(\alpha_i = \alpha_{i+1}\) for some \( 1 \leq i \leq t \). Equivalently, if \(\alpha \neq \underline{\alpha}\), then \(\alpha_i - \alpha_{i+1}>1\) for some \( 1 \leq i \leq t \).
				\end{cor} 
				
				In our example, \(\alpha=(7,7,4,1,1,1) \in \mathcal{ P}(21)\) has diagonal sequence  \[ \, d(\alpha) = (1,2,3,4,4,4,2,1) = (1,2,3,4,4^{(2)},3^{(0)},2^{(1)},1^{(1)})\] 
				and hence \(q=4\), \(s_4=2, s_3=0,s_2=1, s_1=1\). Proposition \ref{prop:s} implies that \(\overline{\alpha}=(8,6,4,3) \in \mathcal{ P}(21)\) while Proposition \ref{prop:alphaunderbar} implies that \(\underline{\alpha}=(4,4,4,3,2,2,1,1) \in \mathcal{ P}(21)\). 
				\begin{center}
					\begin{tabular}[t]{ccccccc}
						\multicolumn{1}{c}{$\alpha$} & &
						\multicolumn{1}{c}{$\alpha^*$} & &
						\multicolumn{1}{c}{$\overline{\alpha}$} & &
						\multicolumn{1}{c}{$\underline{\alpha}$} \\ 
						\topalign{$\young(1111111,2222221,3333,4,4,4)$} & & 
						\topalign{$\young(111111,222,332,432,43,41,21)$} & & 
						\topalign{$\young(11111111,222222,3333,444)$} & &  \topalign{$\young(1111,2221,3321,432,43,41,2,1)$} \\ 
					\end{tabular}
				\end{center}
				The partition \(\overline{\alpha}\) is obtained from the partition \(\alpha\) by moving all squares as far up along their respective diagonals as possible while the partition \(\underline{\alpha}=(4,4,4,3,2,2,1,1)\) is obtained from the partition \(\alpha\) by moving all squares as far down along their respective diagonals as possible.

				Given a diagonal sequence \(d \in \Delta(n)\) and \(\alpha =(\alpha_1, \cdots, \alpha_t) \in [ \, d \, ] \) the values of \(\alpha_1\) are restricted to a certain set. Define \[ A_1=\{q, q+s_q,q+s_q+s_{q-1}, \cdots, q+\sum_{i=1}^q s_i\}.\] 
				
				\begin{prop} \label{prop:alpha1}
					If \( \delta(\alpha) = \left( 1,2, \cdots ,q-1,q,q^{(s_q)},(q-1)^{(s_{q-1})}, \cdots ,1^{(s_1)} \right)\), then \(\alpha_1 \in A_1.\) Equivalently, if \( \delta(\alpha) = \left( 1,2, \cdots ,q-1,q,q^{(s_q)},(q-1)^{(s_{q-1})}, \cdots ,1^{(s_1)} \right)\), then \(\alpha \in \mathcal{P}(n,k)^*\) for some \(k \in A_1\).
				\end{prop}
				
				\begin{proof}
					If \(\alpha=(\alpha_1,\alpha_2,...\alpha_t)\) with \( \delta(\alpha) = \left( 1,2, \cdots ,q-1,q,q^{(s_q)},(q-1)^{(s_{q-1})}, \cdots ,1^{(s_1)} \right)\) let \(\alpha'=(\alpha_2,...\alpha_t)\). The lengths of first \(\alpha_1\) diagonals of \(\alpha'\) decrease by \(1\) while the lengths of the other diagonal stay the same, i.e., \(\delta(\alpha')=(d_k')_{k \geq 1}\) with \(d_i' = d_i -1\) for \(1 \leq i \leq \alpha_1\) and \(d_i' = d_i \) for \(i > \alpha_1\). The constraints imposed by Lemma \ref{lemma:delta} imply that \(\alpha_1 \in \{q, q+s_q,q+s_q+s_{q-1}, \cdots, q+\sum_{i=1}^q s_i\}.\)
				\end{proof}

				Partitions with the same diagonal sequence have the same sum of the squares of their parts and the sum of the squares of the parts of their conjugates.

				For \(\alpha=(\alpha_1,\cdots,\alpha_t) \in \mathcal{P}(n)\) define \(s(\alpha) = \sum_{i=1}^t \alpha_i^2.\) Let \(\alpha^*=\gamma=(\gamma_1,\gamma_2, \cdots,\gamma_s)\), \(s=\alpha_1\). The Young diagram of \(\alpha\) is contained in a \(t\) by \(s\) rectangle. Hence \(d_k=0\) for all \(k \geq t+s\). Let \(r= t+s-1\).  
				
				\begin{prop} \label{lemma:diagonal} If \(\alpha \in \mathcal{P}(n)\), then 
					\(s(\alpha) +s(\alpha^*) =  (2,4, \cdots ,2r) \cdot \delta(\alpha) =2\sum_{k=1}^r k \, d_k.\) As a consequence, if \(\delta(\alpha)=\delta(\beta)\), then \(s(\alpha)+s(\alpha^*) = s(\beta)+s(\beta^*).\)
				\end{prop}
				
				We note, that the converse of the statement is not true. Let \(\alpha=(6,2,1), \beta=(5,4) \in \mathcal{P}(9)\). Then \(\delta(\alpha)=(1,2,3,1,1,1) \neq (1,2,2,2,2)=\delta(\beta) \) and \(s(\alpha)+s(\alpha^*) = s(\beta)+s(\beta^*)=58\).
				
				\begin{proof}
					We use the well-known fact that \(n^2\) is the sum of the first \(n\) odd integers. Consider the square in the \(i\)-th row and \(j\)-th column of the Young diagram of \(\alpha\). Such a square lies on the \(k\)-th diagonal for \(k=i+j-1\). As such it contributes \(2j-1\) to \(\alpha_i^2\) and contributes \(2i-1\) to \(\beta_j^2\). Hence such a square contributes \(2k\) to \(s(\alpha)+s(\alpha^*)\). Every square on the \(k\)-th diagonal contributes the same amount, \(2k\), to \(s(\alpha)+s(\alpha^*)\). Hence the total contributions of all squares on the \(k\)-th diagonal of the Young diagram is \(2k d_k\). Summing over \(k\) yields the result.     
				\end{proof}

				\section{Majorization order on \([ \, d \, ] \)} \label{section:majorization}
				
				The set \(\mathcal{P}(n)\) is a partially ordered set under majorization. Recall that if \(\alpha=(\alpha_1,\alpha_2, \cdots ,\alpha_s) \in \mathcal{P}(n)\) and \(\beta=(\beta_1,\beta_2, \cdots ,\beta_t) \in \mathcal{P}(n)\) then we say \(\alpha\) majorizes \(\beta\) if and only if  \(\sum_{i=1}^k \alpha_i \geq \sum_{i=1}^k \beta_i\) for all  \(1 \leq k \leq \min \{s,t\}.\) If \(\alpha\) majorizes \(\beta\) we write \(\alpha \succ \beta\). The set of partitions \(\mathcal{P}(n)\) is a partially ordered set under majorization as all subsets \(\mathcal{P}(n)\). In particular, we will show that if \(d \in \Delta(n)\), then the partially ordered set \([ \, d \, ] \) has a unique maximal element \(\overline{\alpha}\) and a unique minimal element \(\underline{\alpha}\), where \(\overline{\alpha}\) and \(\underline{\alpha}\) are as defined in Section \ref{sec:properties}.

				\begin{prop}
					Let \(d \in \Delta(n)\). If \(\alpha \in [ \, d \, ]  \subseteq \mathcal{ P}(n)\), then \(\overline{\alpha} \succ \alpha \succ \underline{\alpha}\).
				\end{prop} 
				
				\begin{proof}
					It is well-known (\cite{MOA} Theorem 7.B.5) that for \(\alpha,\beta \in \mathcal{P}(n)\) \(\alpha \succ \beta\) if and only if \(\beta^* \succ \alpha^*.\) Since \(\underline{\alpha} = \overline{\alpha}^*\) we only need to show \(\overline{\alpha} \succ \alpha\) for all \( \alpha \in [ \, d \, ] \) . Let \( \alpha =(\alpha_1,\alpha_2, \cdots ,\alpha_s)\), \(\overline{\alpha}=(\overline{\alpha}_1,\overline{\alpha}_2, \cdots , \overline{\alpha}_t) \in [\,d\,]\). Note that \( \alpha \in [ \, d \, ]  \) if and only if \( \alpha^* \in [ \, d \, ] \). 
					
					Notice that in general a square in the \(i\)-th row of the Young diagram of \(\alpha \in \mathcal{P}(n)\) is the \(j\)-th square of its diagonal with \(j \leq i\). 
					Since \(\overline{\alpha}_i >\overline{\alpha}_{i+1}\) it follows that  every square in the \(i\)-th row of the Young diagram of \(\overline{\alpha}\) is the \(i\)-th square of its respective diagonal. Hence \[\sum_{i=1}^k \overline{\alpha}_i \geq \sum_{i=1}^k \alpha_i \mbox{ for } 1 \leq k \leq \min \{s,t\},\]
					which is what we had to show.
				\end{proof}
				
				\medskip
				
				We can further stratify the set \([ \, d \, ] \) by the number of (non-zero) parts of the partitions and, alternatively, by the size of the largest part of \(\alpha \in [ \, d \, ] \). 
				
				Define \[[ \, d \, ] _k = [ \, d \, ]  \cap \mathcal{P}(n,k) \mbox{ and }[\,d\,]_k^* = [\,d\,] \cap \mathcal{P}(n,k)^*.\] Notice that by Proposition \ref{prop:alpha1} \([ \, d \, ] _k=[ \, d \, ] _k^*=\emptyset \) unless \( k \in A_1\). If \(d=(d_1,d_2,\cdots,d_L)\) and \(\alpha =(\alpha_1,\alpha_2, \cdots,\alpha_k) \in [ \, d \, ] _k\), then \(\alpha'=\alpha_1-1,\alpha_2-1,\cdots,\alpha_k-1 \in \mathcal{P}(n-k)\) with diagonal sequence \[ \, d'=(d_2-1,\cdots, d_k-1,d_{k+1},\cdots, d_L).\] Let \(\overline{\alpha}'=(\overline{\alpha}'_1,\overline{\alpha}'_2,\cdots,\overline{\alpha}'_k)\) be the maximal element in \([\,d'\,]\). Then \(\overline{\alpha}(k)=(\overline{\alpha}'_1+1,\overline{\alpha}'_2+1,\cdots,\overline{\alpha}'_k+1)\) is the maximal element in \([ \, d \, ] _k\). It follows that \(\overline{\alpha}(k)^*\) is the minimal element in \([ \, d \, ] _k^*\). 
				
				Now we describe how to construct the maximal element \(\overline{\gamma} =(\gamma_1,\gamma_2,\cdots ,\gamma_t) \in [\,d\,]_k^*\) for \(k \in A_1\). 
				
				Set \(\gamma_1=k\). Then \( d'=(d_2-1,\cdots, d_k-1,d_{k+1},\cdots, d_L)\) is the diagonal sequence of some partition \(\gamma' \in \mathcal{P}(n-k)\). By Proposition \ref{prop:alpha1} its largest part is restricted to a set \(A_1'\). Set \(\gamma_2\) the largest element in \(A_1'\) less than or equal to \(k\). This process continues in the obvious way and leads to the maximal element \(\overline{\gamma} \in [\,d\,]_k^*\). Hence \(\overline{\gamma}^*=\underline{\alpha}(k)\), the minimal element in \([\,d\,]_k\).

				We illustrate this construction with our example. Let \(d=(1,2,3,4,4,4,2,1)\) and \(k=6\). 
				
				Step 1: Set \(\gamma_1=6\). Then \(d'=(1,2,3,3,3,2,1)\) which implies \(A_1'=\{3,5,6,7\}\). 
				
				Step 2: Set \(\gamma_2=6\). Then \(d''=(1,2,2,2,1,1)\) which implies \(A^{(3)} = \{1,4,6\}.\)  
				
				Step 3: Set \(\gamma_3=6\). Then \(d^{(3)}=(1,1,1)\) which implies \(A_1^{(4)}=\{3\}\). 
				
				Step 4: Set \(\gamma_4=3\) and the process ends.
				
				We get \(\overline{\gamma}=(6,6,6,3)\) and hence \(\underline{\alpha}(6)=(4,4,4,3,3,3).\) 
				
				%\textcolor{red}{Needs work!!!!!!!!}
				
				\section{The cardinality of \([\,d\,]\)} \label{sec:main}

				\textbf{Definition} \textit{Let \(M=\{0^{(b_0)}, 1^{(b_1)}, \cdots ,t^{(b_t)} \} \) be a multiset with \(b_i\) elements equal to \(i\). Let \(b= b_0+b_1+\cdots +b_t\). A vn-arrangement of the elements of \(M\) is a sequence \((v_1,v_2,\cdots,v_b)\) such that \(v_{i+1}-v_i \leq 1\) for \(0 \leq i <b\). }
				
				\begin{prop} \label{prop:vn-arrangements}
					The number of vn-arrangements of \(M\) is \(\prod_{i=0}^{t-1} {b_i+b_{i+1} \choose b_i}\).
				\end{prop}
				
				\begin{proof}
					For \(t=0\) there is nothing to prove. When \(t=1\), any arrangement of the \(b_0\) elements equal to 0 and the \(b_1\) elements equal to 1 are vn-arrangements. There are \({b_0+b_1 \choose b_0}\) such sequences. So assume the results holds for \(t \geq 1\). Let \( M=\{0^{(b_0)}, 1^{(b_1)}, \cdots ,t^{(b_t)},(t+1)^{(b_{t+1})} \}\). Any vn-arrangement of \(M\) arises from a vn-arrangement of \(M'=\{0^{(b_0)},1^{(b_1)}, \cdots ,t^{(b_t)}\}\) by adding any number of elements equal to \(t+1\) at the very beginning of the arrangement or after an element equal to \(t\). Since there are \(b_{t+1}\) elements equal to \(t+1\), which can be placed in \(b_t+1\) buckets, there are \({b_t+b_{t+1} \choose b_t}\) vn-arrangements of \(M\) for every vn-arrangement of \(M'\). By induction, the results holds.    
				\end{proof}
				
				The result generalizes and holds for any finite multiset \( \{a^{(b_0)}, (a+1)^{(b_1)}, \cdots ,(a+t)^{(b_t)} \} =a+M\), \(a \in \mathbb{Z}\). It is worth pointing out the special case \(b_i=1\) for all \(1 \leq i \leq t\).
				
				\begin{cor}
					If \(M\) is any set of consecutive integers, then the number of vn-arrangements of \(M\) is \(2^{t-1}\). 
				\end{cor}
				
				%\begin{proof}
				%    If \(b_i=1\) for all \(1 \leq i \leq b\), then  \({b_i+b_{i+1} \choose b_i} =2\) %for all \(1 \leq i <b\). 
				%\end{proof}
				
				The idea of vn-arrangements is related to the beautiful theory of counting sequences according to rises and falls, which was developed by Carlitz and others. E.g., see \cite{C} and \cite{A}. As an aside, we can extend the definition of vn-arrangement as follows. 
				
				\textbf{Definition:} \textit{Let \(M=\{0^{(b_0)}, 1^{(b_1)},2^{(b_2)}, \cdots ,t^{(b_t)} \} \) be a multiset with \(b_i\) elements equal to \(i\). Let \(b= b_0+b_1+\cdots +b_t\). A k-vn-arrangement, \(0 \leq k \leq t\), of the elements of \(M\) is a sequence \((v_1,v_2,\cdots,v_b)\) such that \(v_{i+1}-v_i \leq k\) for \(1 \leq i <b\). }
				
				There is an explicit formula for the number of k-vn-arrangments of a given multiset. We adopt the convention that an empty product has value 1.  
				
				\begin{prop}
					If \(M=\{0^{(b_0)}, 1^{(b_1)},2^{(b_2)}, \cdots ,t^{(b_t)} \} \) be a multiset with \(b_i\) elements equal to \(i\). 
					The number of k-vn-arrangements of \(M\) is \( {b_0+b_1+\cdots +b_k \choose b_0,b_1,\cdots, b_k}  \,\left(\prod_{i=1}^{t-k} {b_i+b_{i+1}+\cdots + b_{i+k} \choose b_{i+k}} \right).\)
				\end{prop}
				
				\begin{proof}
					For \(t < k\), there are no constraints on the arrangements. The number of arrangement is the multinomial coefficient. So assume now \(t \geq k\).  
					
					For \(k=0\) there is only one arrangement, \((t^{(b_t)},(t-1)^{(b_{t-1})}, \cdots 1^{(b_1)}, \cdots ,0^{(b_0)})\) and the result holds. For \(k=1\) we get the result of Proposition \ref{prop:vn-arrangements}. Now assume \(k \geq 2\). We proceed by induction on \(t \geq k\). If \(t=k\), the results holds (with the convention that an empty product is equal to 1). Now assume the result holds for some \(t \geq k\). Let \( M=\{0^{(b_0)}, 1^{(b_1)}, \cdots , t^{(b_t)},(t+1)^{(b_{t+1})} \}\). Any k-vn-arrangement of \(M\) arises from a k-vn-arrangement of \(M'=\{0^{(b_0)},1^{(b_1)}, \cdots ,t^{(b_t)}\}\) by adding any number of elements equal to \(t+1\) at the very beginning or after an element equal to \(t+1-k,t+2-k,\cdots ,t\). Since there are \(b_{t+1}\)  elements equal to \(t+1\), which can be placed in \(b_{t+1-k}+b_{t+2-k}+ \cdots + b_{t}\) buckets, the result follows by induction.    
				\end{proof}
				
				For a partition \(\alpha=(\alpha_1,...,\alpha_t) \in \mathcal{ P}(n)\) define \(v(\alpha)=(v_1,v_2,\cdots,v_t)\) where \(v_i=\alpha_i+i-1\) for \(1 \leq i \leq t\). The definition of \(\delta(\alpha)=(d_k)_{k \geq 1}\) can now be restated as \begin{equation} \label{eqn:vi} d_k = \big\lvert \{ i \, \rvert \, 1 \leq i \leq k \mbox{ and  } v_i \geq k \}  \big\rvert.
				\end{equation} 
				Since \(v_i=\alpha_i+i-1\) we have \(v_i \geq i\) for \(1 \leq i \leq t\) and since \(\alpha\) is a non-increasing sequence we have \(v_{i+1}-v_i \leq 1\) for \(1 \leq i < L\). In the special case of \(\overline{\alpha}\) we have \(\overline{v}_i = q+\sum_{k=i}^{q} s_k\) for \(1 \leq i \leq q\). In particular, \(\overline{v_1}=q+\sum_{i=1}^q s_i=L\) and \(\overline{v}_q =q+s_q=l\).  
				
				In what follows we assume \(\delta=(1,2,\cdots,q,d_{q+1},d_{q+2}, \cdots, d_L) \in \Delta(n)\), \(q \geq d_{q+1}\). Set \(b_i=d_i-d_{i+1}\) for \(q \leq i <L\) and \(b_L=d_L\). By Proposition \ref{prop:alpha1}, if \(\alpha \in [ \, d \, ] \), then \(\alpha \in \mathcal{P}(n,k)\) for some \(k \in A_1\). Let \([\,d\,]_k= [\,d\,] \cap \mathcal{P}(n,k)\). Our first result counts the number of elements in \([ \, d \, ] _q\), i.e., those partitions in \(\mathcal{P}(n)\) that have exactly \(q\) parts. 
				
				\begin{prop} \label{prop:lengthq}
					\begin{equation}
						\big\rvert \, [ \,d \, ]_q| = \big\rvert \, [ \, d \, ]  \cap \, \mathcal{P}(n,q) \, \big\rvert= \prod_{i=q}^{L-1} {b_i+b_{i+1} \choose b_i}. 
					\end{equation}
				\end{prop}
				
				\begin{proof}
					Let \(\alpha = (\alpha_1,\alpha_2,\cdots ,\alpha_q) \in [d]_q\) and \(v(\alpha)=(\alpha_1,\alpha_2+1,\cdots,\alpha_q+q-1)\). We will show that \(v(\alpha)\) is a vn-arrangement of the multiset \(\{\overline{v}_1,\overline{v}_2,\cdots, \overline{v}_q\}.\) Notice that \(v_{i+1}-v_i \leq 1\) since \(\alpha_i \leq \alpha_{i+1}\) for all \(1 \leq i < q\), i.e., \(v(\alpha)\) is a vn-arrangement. Since \(\alpha,\overline{\alpha} \in [ \, d \, ] \) Equation \ref{eqn:vi} implies that 
					\begin{equation}
						d_k = \big\lvert \{ i \, \rvert \, 1 \leq i \leq k \mbox{ and  } v_i \geq k \}  \big\rvert= \big\lvert \{ j \, \rvert \, 1 \leq j \leq k \mbox{ and  } \overline{v_j} \geq k \}  \big\rvert \mbox{ for all } q \leq k.
					\end{equation}
					Since \(d_k =0\) for \(k >L\) and \(\sum_{i=1}^q v_i = \sum_{j=1}^q \overline{v}_j\) we conclude that \(\{v_1,v_2,\cdots, v_q\} = \{ \overline{v}_1,\overline{v}_2,\cdots,\overline{v}_q\}\) as multisets. 
					
					Writing the elements of \(\overline{v}\) as a multiset of consecutive integers we have \[M=\{q^{b_q}, \cdots, l^{b_l},(l+1)^{b_{l+1}}, \cdots, L^{b_L}\}.\] By Proposition \ref{prop:alpha1}, \(b_i \neq 0\) if and only if \(i \in A_1\). The result now follows from Proposition \ref{prop:vn-arrangements}.
				\end{proof}

				%\begin{prop} \label{prop:lengthq+1}
				%    \begin{equation}
					%      | \, [ \, d \, ]  \cap \, \mathcal{P}(n,q+1) \,|= \dfrac{b_{q+1}}{b_q+1} %\prod_{i=q}^{L-1} {b_i+b_{i+1} \choose b_i}. 
					%    \end{equation}
				%\end{prop}
				
				The expression for the cardinalities of \(\big\rvert \, [ \, d \, ]_k \big\rvert \) when \(k >q\) are more complicated. 
				
				\begin{prop} \label{prop:lengthq+2} 
					If \(k > q\), then 
					\begin{equation} \label{eqn:lengthk}   
						\big\rvert \, [ \, d \, ]_k \, \big\rvert = \big\rvert \, [ \, d \, ]  \cap \, \mathcal{P}(n,k) \, \big\rvert= {b_{k-1}+b_k \choose b_{k-1}+1}\prod_{i=q}^{k-2} {b_i+b_{i+1}+1 \choose b_i+1} \prod_{i=k}^{L-1} {b_i+b_{i+1} \choose b_i}  
					\end{equation}
				\end{prop}
				
				As usual, an empty product has value 1. Notice that for \(b_k \neq 0\) the binomial coefficient \({b_{k-1}+b_k \choose b_{k-1}+1} = 0\). Hence \(\big\rvert \, [ \, d \, ]_k \, \big\rvert \neq 0 \) if and only if \(k \in A_1\).   
				
				\begin{proof} 
					We proceed by induction on \(n\). The result is straightforward to verify for small values of \(n\). Let \(d=(1,2, \cdots q,d_{q+1} \cdots, d_L)\), be a diagonal sequence of some partition \(\alpha \in \mathcal{P}(n)\) with \(k \in A_1\). If \(k \not\in A_1\), there is nothing to prove. Since \(k \in A_1\), there exists a partition \(\alpha =(\alpha_1,\cdots \alpha_k)\in \mathcal{P}(n)\) with \(\alpha_k >0\) and \(\delta(\alpha)=d\). Hence the partition \(\alpha'=(\alpha_1-1\,\cdots \alpha_k-1) \in \mathcal{P}(n-k)\) and \(\delta(\alpha')=d'=(d_i')_{1 \leq i \leq L-1}=(1,2,\cdots, q-1,d_q-1,\cdots, d_{k-1}-1,d_{k},\cdots, d_{L-1})\). 
					
					If \(\beta=(\beta_1,\cdots \beta_l) \in [\,d'\,]_l\) for some \(l \leq k\), then \((\beta_1+1,\cdots \beta_l+1,1^{(k-l)}) \in [\,d\,]_k\) and vice versa. Hence \[ \big\rvert \, [\, d\,]_k \, \big\rvert = \sum_{l=q-1}^k \big\rvert \, [ \, d' \, ]_l \, \big\rvert. \]
					
					We observe that for \(i \geq k-1\), \(c_{k-1}=d_k'-d_{k-1}'=d_{k+1}-d_k-1=b_k-1\) while for \(i \neq k-1\), \(c_i = d_{i+1}'-d_i'=d_{i+2}-d_{i+1}=b_{i+1}.\)
					
					By induction \[ \big\rvert \, [ \, d' \, ]_q \, \big\rvert =  \prod_{i=q-1}^{L-2} {c_i+c_{i+1} \choose c_i}\] and 
					\[ \big\rvert \, [ \, d' \, ]_l \, \big\rvert = {c_{l-1}+c_l \choose c_{l-1}+1}\prod_{i=q-1}^{l-3} {c_i+c_{i+1}+1 \choose c_i+1} \prod_{i=l-1}^{L-2} {c_i+c_{i+1} \choose c_i}\] for \(l >q-1\). 
					
					Notice that \begin{align*} \big\rvert \, [ \, d' \, ]_{q-1} \, \big\rvert + \big\rvert \, [ \, d' \, ]_{q} \, \big\rvert & =\prod_{i=q-1}^{L-2} {c_i+c_{i+1} \choose c_i}+{c_{q-1}+c_{q} \choose c_{q-1}+1}\prod_{i=q-1}^{q-2} {c_i+c_{i+1}+1 \choose c_i+1} \prod_{i=q}^{L-2} {c_i+c_{i+1} \choose c_i} \\ & =
						{c_{q-1}+c_{q}+1 \choose c_{q-1}+1} \prod_{i=q}^{L-2} {c_i+c_{i+1} \choose c_i} \end{align*} and by induction \begin{align*} \big\rvert \, [\, d\,]_k \, \big\rvert = \sum_{l=q-1}^{k} \big\rvert \, [ \, d' \, ]_l \, \big\rvert & = \prod_{i=q}^{k}{c_{i-1}+c_{i}+1 \choose c_{i-1}+1} \prod_{i=k}^{L-2} {c_i+c_{i+1} \choose c_i} \\ & =  {b_{k-1}+b_k \choose b_{k-1}+1}\prod_{i=q}^{k-2} {b_i+b_{i+1}+1 \choose b_i+1} \prod_{i=k}^{L-1} {b_i+b_{i+1} \choose b_i}.  
					\end{align*}
				\end{proof}
				
				The cardinality of \([ \, d \,]\) now follows by addition of the cardinalities of \( [\,d\,]_k\) for \(q \leq k \leq L\) with a similar and somwhat simpler argument.
				
				\begin{thm} \label{theorem:anylength}
					Assume \(\delta=(1,2,\cdots,q,d_{q+1},d_{q+2}, \cdots, d_L) \in d(n)\). Set \(b_i=d_i-d_{i+1}\) for \(q \leq i <L\) and \(b_L=d_L\). The number of partitions \(\alpha \in \mathcal{ P}(n)\) with \(d(\alpha)=d\) is 
					\begin{equation} \label{eqn:anylength}
						\big\rvert \, [ \,d\,]\, \big\rvert= \prod_{i=q}^{L-1} {b_i+b_{i+1}+1 \choose b_i+1}. 
					\end{equation}
				\end{thm}
				
				\begin{proof} Clearly \(\big\rvert \, [ \, d \, ] \, \big\rvert =        \sum_{k=q}^L \big\rvert \, [ \,d \,]_k \, \big\rvert\) and         \[\sum_{k=q}^L \big\rvert \, [ \,d \,]_k \, \big\rvert  = 
					\prod_{i=q}^{L-1} {b_i+b_{i+1} \choose b_i}+\sum_{k=q+1}^L {b_{k-1}+b_k \choose b_{k-1}+1}\prod_{i=q}^{k-2} {b_i+b_{i+1}+1 \choose b_i+1} \prod_{i=k}^{L-1} {b_i+b_{i+1} \choose b_i}.\] Notice that \begin{align*} \big\rvert \, [ \, d \, ]_q \, \big\rvert + \big\rvert \, [ \, d \, ]_{q+1} \, \big\rvert & =\prod_{i=q}^{L-1} {b_i+b_{i+1} \choose b_i}+{b_{q}+b_{q+1} \choose b_{q}+1}\prod_{i=q}^{q-1} {b_i+b_{i+1}+1 \choose b_i+1} \prod_{i=q+1}^{L-1} {b_i+b_{i+1} \choose b_i} \\ & =
						{b_{q}+b_{q+1}+1 \choose b_{q}+1} \prod_{i=q+1}^{L-1} {b_i+b_{i+1} \choose b_i} 
					\end{align*}
					and by induction \begin{align*} \big\rvert \, [ \, d \, ] \, \big\rvert =    
						\sum_{k=q}^L \big\rvert \, [ \,d \,]_k \, \big\rvert = \prod_{i=q}^{L-1} {b_i+b_{i+1}+1 \choose b_i+1}.
					\end{align*}
				\end{proof}

				We point out some corollaries to the main result. 
				
				\begin{cor} \label{cor:si>1}
					\begin{enumerate}
						\item If \(s_i \geq 2\) for all \(1 \leq i < q\), then \(\big\rvert \, [ \,d \, ]_q \big\rvert = 1. \) %In particular, \(\overline{\alpha}=\underline{\alpha}_q\). 
						\item If \(s_i \geq 2\) for all \(1 \leq i < q\) and \(s_q \geq 1\), then \(\big\rvert \, [ \,d \, ]\, \big\rvert= \prod_{i=q}^{L} (b_i+1). \)
						\item If \(m\) is a natural number then there exists partition  \(\alpha \in \mathcal{P}(n)\) for some \(n\) such that \(\delta(\alpha)=d\) and \(|\, [\,d\,]\,|=m\).
						\item  Let \(d=(1,2,\cdots,q,k^{(s_k)},(k-1)^{(s_{k-1})},\cdots,2^{(s_2)},1^{(s_1)})\) with \(k<q\) and \(s_k \geq 2\) and define \(d'=(1,2, \cdots,k,k^{(s_k)},(k-1)^{(s_{k-1})}, \cdots,2^{(s_2)},1^{(s_1)})\). Then \(\big\rvert \, [\, d \,] \, \big\rvert = \big\rvert \, [\, d' \,] \, \big\rvert\).
						\item  Let \(d=(1,2,\cdots,q,q^{(s_q)},(q-1)^{(s_{q-1})},\cdots,2^{(s_2)},1^{(s_1)})\). Let \(\sigma_i=\min \{s_i,2\}\) for \(1 \leq i \leq q\) and define \(d'=(1,2,\cdots,q,q^{(\sigma_q)},(q-1)^{(\sigma_{q-1})},\cdots,2^{(\sigma_2)},1^{(\sigma_1)})\). Then \(\big\rvert \, [\,d\,] \, \big\rvert = \big\rvert \, [\,d'\,] \, \big\rvert\).
					\end{enumerate}
				\end{cor}
				
				\begin{proof}
					\begin{enumerate}
						\item Part 1 follows from the fact that if \(s_k \geq 2\) for all \(1 \leq k < q\), then then \(b_i=0\) or \(b_{i+1}=0\) for all \(q \leq i <L\). Hence \({b_i + b_{i+1} \choose b_i} =1\) for all \(q \leq i <L\).
						
						\item For the proof of part 2 assume\(s_k \geq 2\) for \(1 \leq k \leq q\) and \(s_q\geq 1\). This implies that if \(b_i \neq 0\), then \(b_{i-1}=b_{i+1}=0\). Hence \[{b_{i-1}+b_i+1 \choose b_{i-1}+1}{b_i+b_{i+1}+1 \choose b_i+1} = (b_i+1) \cdot 1 =b_i+1. \]
						
						\item Part 3 follows from part 2 by factoring \(m\). 
						
						\item Part 4 follows from the fact that \(s_k\geq 2\) implies \(b_{q+1}=0\) and hence \({b_q+b_{q+1} \choose b_q+1}=1\) for any choice of \(b_q\).
						
						\item Part 5 follows the fact that if \(s_i >2\), then \(b_j=b_{j+1}=0\) for some \(j\) and hence \({b_j+b_{j+1}+1 \choose b_j+1}=1\).  
					\end{enumerate}
				\end{proof}
				
				Given Equation \ref{eqn:anylength} we can characterize diagonal sequences \(d \in \Delta(n)\) and the sets \([ \, d \, ] \) for which \(\big\rvert \, [ \, d \, ] \, \big\rvert\) is small or a prime number.  
				
				\begin{cor}
					\begin{enumerate}
						\item If \(d \in \Delta(n)\) and \(\big\rvert \, [ \,d \, ]\, \big\rvert=1\), then there exists a positive integer \(q\) such that \(n={q+1 \choose 2}\), \(d=(1,2,\cdots,q-1,q)\) and \([ \, d \, ] =\{(q,q-1,\cdots,2,1)\}.\)
						\item If \(\big\rvert \, [ \,d \, ]\, \big\rvert=2\), then there exists a positive integers \(q\) and \(k \geq 2\) such that \(n={q+1 \choose 2}+k\), \(d=(1,2,\cdots,q,1^{(k)}) \) and  \([ \, d \, ] =\{ (q+k,q-1,\cdots,2,1),(q,q-1,\cdots,2,1,1^{(k)})\}.\)
						\item If \(\big\rvert \, [ \,d \, ]\, \big\rvert=3\), then 
						\begin{enumerate}
							\item there exist integers \(q \geq 2\), \(k \geq 2\) and \(d=(1,2,\cdots,q,2^k)\). In this case \( [ \, d \, ]  =\{ (q+k,q-1+k,q-2, \cdots, 2,1),\)\( (q,q-1,\cdots,2^{(k+1)},1), \)\((q+k,q-1,\cdots,2,1^{(k+1)}\}\); or
							\item \(d=(1,2,1)\) and \([\,d\,]=\{(3,1),(2,2),(2,1,1)\}\). 
						\end{enumerate}   
						\item If \(\big\rvert \, [ \,d \, ]\, \big\rvert=4\), then 
						\begin{enumerate}
							\item there exist integers \(q \geq 3\), \(k \geq 2\), \(d=(1,2,\cdots,q,3^{(k)})\) and \([\,d\,]=\{(q+k,q-1+k,q-2+k,q-3,\cdots,2,1),\)\((q+k,q-1+k,q-2,\cdots,2,1^{(k+1}),\)\((q+k,q-1,\cdots,3,2^{(k+1)},1),\)\((
							q.q-1,\cdots,4,3^{(k+1)},2,1)\}\);  or
							\item \([ \, d \, ] =(1,2,3,3) \) and \([\,d\,] = \{(4,3,2),(4,3,1,1),(4,2,2,1),\)\((3,3,2,1)\}\); or 
							\item there exist integers  \(q,k,l \geq 2\), \(d=(1,2,\cdots,q,2^{(k)},1^{(l)}) \) and 
							\([\,d\,]=\{(q+k+l,q+k,q-2,\cdots,2,1)\),
							\((q+k+l,q-1,\cdots 2,1^{(k+1)}),\)
							\((q+k,q-1,\cdots,2,1^{(k+l+1)}),\)
							\((q,\cdots,3,2^{(k+1)},1^{(l+1)})\}\); or
							\item there exist an integer \(k \geq 2\), \(d=(1,2,2,1^{(k)})\) and \([\,d\,]=\{ (k+3,2),\) \((k+3,1,1),\) \((3,1^{(k+2)}),\)\((2,2,1^{(k+1)})\}\).
						\end{enumerate}
					\end{enumerate}
				\end{cor}
				
				We leave it to the reader to generalize parts 2 and 3 and find all partitions \(\alpha\) and \(d=\delta(\alpha)\) such that \(\big\rvert \, [\,d\,] \, \big\rvert=p\), \(p\) prime. 
				
				\begin{proof}
					\begin{enumerate}
						\item If  \(\big\rvert \, [ \,d \, ]\, \big\rvert=1\), then \({b_i+b_{i+1}+1 \choose b_i+1}=1\) for \(q \leq i < L\), which implies \(b_i=0\) for \(q< i \leq L\) and \(b_q=q\). Hence \(d=(1,2,\cdots,q)\) and \([ \, d \, ] =\{(q,q-1,\cdots,2,1)\}.\) 
						\item If \(\big\rvert \, [ \,d \, ]\, \big\rvert=2\), then there exists a \(q \leq j < L\) with \({b_j+b_{j+1}+1 \choose b_j+1} = 2\) while \({b_i+b_{i+1}+1 \choose b_i+1} = 1\) for all \(q \leq i \neq j<L. \) Hence \(b_L=1\) while \(b_i=0\) for all \(q < i < L\), which implies \(d=(1,2,\cdots,q,1,1,\cdots, 1)\) and \(n={q+1 \choose 2} +k\) for some positive integer \(k\). It follows that \([ \, d \, ] =\{ (q+k,q-1,\cdots,2,1),(q,q-1,\cdots,2,1,1^{(k)}) \}\).  
						\item If \(\big\rvert \, [ \,d \, ]\, \big\rvert=3\), then \({b_{L-1}+b_L+1 \choose b_{L-1}+1}=3\) while \({b_{j}+b_{j+1}+1 \choose b_j+1}=1\) for all \(q \leq j < L-1\), which implies there exists integer \(q\geq 2\), \(k>1\) such that \(d=(1,2,\cdots,q,2^{(k)})\) or \(d=(1,2,1)\). The result follows.         
						\item If \(\big\rvert \, [ \,d \, ]\, \big\rvert = 4\), then \({b_{L-1}+b_{L}+1 \choose b_{L-1}+1}=4\), \(L \geq q+2\), while \({b_{i}+b_{i+1}+1 \choose b_{i}+1}=1\) for \(q<i<L-1\) or there exist integers \(k \geq q+2\), \(L \geq k=2\) with \({b_{k}+b_{k+1}+1 \choose b_{k}+1}=2\) and \({b_{L-1}+b_{L}+1 \choose b_{L-1}+1}=2\) while \({b_{i}+b_{i+1}+1 \choose b_{i}+1}=1\) for all \(i >q\), \(i \neq k,L-1\). The result follows.     
					\end{enumerate}
				\end{proof}

				\section{Examples} \label{sec:examples}
				We illustrate the results using our example \(\delta=(1,2,3,4,4,4,2,1)\). It follows that \(q=4\) and \(b_4=b_5=0,b_6=2,b_7=8=1\). By Proposition \ref{prop:alpha1}, if \(\alpha \in [ \, d \, ]  \), then \(\alpha \in [ \, d \, ] _4 \cup  [ \, d \, ] _6 \cup  [ \, d \, ] _7 \cup [ \, d \, ] _8. \) By Proposition \ref{prop:lengthq}, Proposition \ref{prop:lengthq+2} and Theorem \ref{theorem:anylength} we get 
				
				\begin{align*}
					\big\rvert \, [ \, d \, ] _4  \, \big\rvert & ={0+0 \choose 0}{0+2 \choose 0}{2+1 \choose 2}{1+1 \choose 1} = 1 \cdot 2 \cdot 3 \cdot 1 = 6,\\
					& \\
					\big\rvert \, [ \, d \, ] _6 \, \big\rvert & ={2 \choose 1}{1 \choose 1}{3 \choose 1}{2 \choose 1} = 2 \cdot 1 \cdot 3 \cdot 2 = 6,\\\
					& \\
					\big\rvert \, [ \, d \, ] _7  \, \big\rvert & ={3 \choose 3}{1 \choose 1}{3 \choose 1}{2 \choose 1} = 1 \cdot 1 \cdot 3 \cdot 2 = 6,\\\
					& \\
					\big\rvert \, [ \, d \, ] _8 \, \big\rvert & ={2 \choose 2}{1 \choose 1}{3 \choose 1}{4 \choose 3} = 1 \cdot 1 \cdot 3 \cdot 4 = 12.\\
				\end{align*} 
				and finally \[ \big\rvert \, [ \, d \, ] \, \big\rvert = {0+0+1 \choose 0+1}{0+2+1 \choose 0+1}{2+1+1 \choose 2+1}{1+1+1 \choose 1+1} = 1 \cdot 3 \cdot 4 \cdot 3 = 36.\]
				
				The 36 partition \(\alpha \in [ \, d \, ]  = [ \, d \, ] _4 \cup [ \, d \, ] _6 \cup [ \, d \, ] _7 \cup [ \, d \, ] _8 \) are listed below . 
				$$
				\begin{array}{ll|ll}
					4A & (8,6,4,3) &  7A & (8,5,4,1,1,1,1) \\ 
					4B & (8,5,5,3) &  7B & (8,5,2,2,2,1,1) \\ 
					4C & (8,5,4,4) &  7C & (8,3,3,3,2,1,1) \\ 
					4D & (7,7,4,3) &  7D & (6,5,2,2,2,2,2) \\ 
					4E & (6,6,6,3) &  7E & (6,3,3,3,2,2,2) \\ 
					4F & (6,5,5,5) &  7F & (4,4,4,3,2,2,2) \\ \hline
					6A & (8,6,4,1,1,1) &  8A & (7,5,4,1,1,1,1,1) \\ 
					6B & (8,6,2,2,2,1) &  8B & (7,5,2,2,2,1,1,1) \\ 
					6C & (8,5,5,1,1,1) &  8C & (7,3,3,3,2,1,1,1) \\ 
					6D & (8,5,2,2,2,2) &  8D & (6,6,4,1,1,1,1,1) \\ 
					6E & (8,3,3,3,3,1) &  8E & (6,6,2,2,2,1,1,1) \\ 
					6F & (8,3,3,3,2,2) &  8F & (6,5,5,1,1,1,1,1) \\ 
					6G & (7,7,4,1,1,1) &  8G & (6,5,2,2,2,2,1,1) \\ 
					6H & (7,7,2,2,2,1) &  8H & (6,3,3,3,3,1,1,1) \\ 
					6I & (6,6,6,1,1,1) &  8I & (6,3,3,3,2,2,1,1) \\ 
					6J & (6,3,3,3,3,3) &  8J & (4,4,4,4,2,1,1,1) \\ 
					6K & (4,4,4,4,4,1) &  8K & (4,4,4,3,3,1,1,1) \\ 
					6L & (4,4,4,3,3,3) &  8L & (4,4,4,3,2,2,1,1) \\ 
				\end{array}
				$$
				
				Alternatively, we can collect the elements of \([ \, d \, ] =[\,d\,]^* [ \, d \, ] _4^* \cup [ \, d \, ] _6^* \cup [ \, d \, ] _7^* \cup [ \, d \, ] _8^*\) by the the size of \(\alpha_1 \in A_1\). 
				$$
				\begin{array}{ll|ll}
					4A^* &  (4,4,4,3,2,2,1,1) &  7A^* &  (7,3,3,3,2,1,1,1) \\ 
					4B^* &  (4,4,4,3,3,1,1,1) &  7B^* &  (7,5,2,2,2,1,1,1) \\ 
					4C^* &  (4,4,4,4,2,1,1,1) &  7C^* &  (7,5,4,1,1,1,1,1) \\ 
					4D^* &  (4,4,4,3,2,2,2) &  7D^* & (7,7,2,2,2,1)  \\ 
					4E^* & (4,4,4,3,3,3) &  7E^* & (7,7,4,1,1,1)  \\ 
					4F^* & (4,4,4,4,4,1) &  7F^* & (7,7,4,3)  \\ \hline
					6A^* & (6,3,3,3,2,2,1,1)  &  8A^* &  (8,3,3,3,2,1,1) \\ 
					6B^* & (6,5,2,2,2,2,1,1) &  8B^* &  (8,5,2,2,2,1,1) \\ 
					6C^* & (6,3,3,3,3,1,1,1) &  8C^* &  (8,5,4,1,1,1,1) \\ 
					6D^* & (6,6,2,2,2,1,1,1) &  8D^* & (8,3,3,3,2,2)  \\ 
					6E^* & (6,5,5,1,1,1,1,1)  &  8E^* & (8,5,2,2,2,2)  \\ 
					6F^* & (6,6,4,1,1,1,1,1)  &  8F^* & (8,3,3,3,3,1)  \\ 
					6G^* & (6,3,3,3,2,2,2) &  8G^* & (8,6,2,2,2,1)  \\ 
					6H^* & (6,5,2,2,2,2,2) &  8H^* & (8,5,5,1,1,1)  \\ 
					6I^* & (6,3,3,3,3,3)  &  8I^* & (8,6,4,1,1,1)  \\ 
					6J^* & (6,6,6,1,1,1) &  8J^* & (8,5,4,4)  \\ 
					6K^* & (6,5,5,5) &  8K^* & (8,5,5,3) \\ 
					6L^* & (6,6,6,3)  &  8L^* & (8,6,4,3)  \\ 
				\end{array}
				$$
				
				Furthermore, \[\begin{array}{ll} 
					\overline{\alpha}(4) = \overline{\alpha}=(8,6,4,3) & \underline{\alpha}(4)=(6,5,5,5) \\
					\overline{\alpha}(6) = (8,6,4,1,1,1) & \underline{\alpha}(6)=(4,4,4,3,3,3) \\
					\overline{\alpha}(7) = (8,5,4,1,1,1,1) & \underline{\alpha}(7) =(4,4,4,3,2,2,2) \\
					\overline{\alpha}(8) = (7,5,4,1,1,1,1,1) & \underline{\alpha}(8)= \underline{\alpha} = (4,4,4,3,2,2,1,1)
				\end{array}\]

			\end{document}